\documentclass{article}
\usepackage{graphicx} 

\usepackage{a4wide}
\usepackage[utf8]{inputenc}
\usepackage[a4paper, margin=2.3cm]{geometry}
\usepackage{amsmath,amssymb,amsthm}
\usepackage{color}
\usepackage{xcolor}
\usepackage{parskip}
\usepackage{hyperref}
\usepackage{parskip}
\usepackage{setspace}
\usepackage{enumitem}
\usepackage{comment}
\usepackage{float} 

\usepackage{geometry}
\usepackage{booktabs}
\newtheorem{theorem}{Theorem}[section]
\newtheorem{remark}{Remark}[section]
\newtheorem{corollary}[theorem]{Corollary}

\newtheorem{lemma}[theorem]{Lemma}

\newtheorem{proposition}[theorem]{Proposition}
\newtheorem{definition}[theorem]{Definition}

\newtheorem*{definition*}{Definition}

\title{A sharp point-sphere incidence bound for $(u, s)$-Salem sets}

\author{
Steven Senger\thanks{Department of Mathematics, Missouri State University. Email: StevenSenger@MissouriState.edu}\and Dung The Tran\thanks{VNU University of Science. ~~~~~~~~~~~~~~~~~~~~~~~~~~~~~~~~~~~~ Email: tranthedung56@gmail.com }
  }
\date{}

\begin{document}

\maketitle
\begin{abstract}
We establish a sharp point-sphere incidence bound in finite fields for point sets exhibiting controlled additive structure.
Working in the framework of \((4,s)\)-Salem sets, which quantify pseudorandomness via fourth-order additive energy, we prove that if
\(P\subset \mathbb{F}_q^d\) is a \((4,s)\)-Salem set with \(s\in \big( \frac{1}{4}, \frac{1}{2} \big]\) and \(|P|\ll q^{ \frac{d}{4s}}\), then for any finite family \(S\) of spheres in \(\mathbb{F}_q^d\),
\[
\bigg| I(P,S)-\frac{|P||S| }{q} \bigg| \ll q^{\frac{d}{4}}\,|P|^{1-s}\,|S|^{\frac{3}{4}}.
\]
This estimate improves the classical point-sphere incidence bounds for arbitrary point sets across a broad parameter range. The proof combines additive energy estimates with a lifting argument that converts point-sphere incidences into point-hyperplane incidences in one higher dimension while preserving the \((4,s)\)-Salem property. As applications, we derive refined bounds for unit distances and sum-product type phenomena, and we extend the method to \((u,s)\)-Salem sets for even moments \(u\ge4\).
\end{abstract}



\tableofcontents

\section{Introduction}
Throughout the paper, for a set $A\subset\mathbb{F}_q^d,$ we write its indicator function $1_A(x)$ as simply $A(x)$ and the notation $X \ll Y$ (equivalently, $X = O(Y)$)
means that there exists an absolute constant $C>0$ such that $X \leq C Y$. Also, in what follows, we may use $0$ to denote the appropriate zero vector. We begin by giving a definition of the Fourier transform that will serve our purposes.

\begin{definition}[Fourier transform and norms]
Let $A\subset\mathbb{F}_q^d$. The Fourier transform of the indicator function of $A$ (often just called the Fourier transform of $A$ when context is clear) is defined by
\[
\widehat{A}(x)
:= q^{-d}\sum_{y\in A} \chi(-x\cdot y),
\qquad x\in\mathbb{F}_q^d,
\]
where $\chi(\cdot)$ is a fixed nontrivial additive character on $\mathbb{F}_q$.
\end{definition}

For $u\in[1,\infty]$, we define the normalized $L^u$-norm of the Fourier transform of $A$ by
\[
\|\widehat{A}\|_u :=\left(\frac{1}{q^d}\sum\limits_{x\in\mathbb{F}_q^d\setminus\{0\}}
|\widehat{A}(x)|^u\right)^{\frac{1}{u}},
\]
when $u\in[1,\infty),$ and
\[
\|\widehat{A}\|_u :=\sup\limits_{x\in\mathbb{F}_q^d\setminus\{0\}}
|\widehat{A}(x)|,
\]
when $u=\infty.$

In \cite{iosevich-Rudnev-07}, Iosevich and Rudnev defined \emph{Salem sets} in finite fields to be those sets $A\subset \mathbb F_q^d$ satisfying the optimal bound on Fourier decay:
\[
\max_{m\neq 0}|\widehat{A}(m)|
\;\le\; q^{-d}|A|^{\frac{1}{2}}.
\]
The restriction that $m\neq 0$ is because we trivially have that $\widehat{A}(0) = q^{-d}|A|$ regardless of the structure of $A,$ so it provides no relevant information. In particular, they showed that these sets satisfy expected behaviors with respect to the distance problem they explore in \cite{iosevich-Rudnev-07}. More recently, Fraser \cite{fraser2} introduced a relaxed notion of these called $(u,s)$-Salem sets.


\begin{definition}\label{def:us-salem}
Let $u\in[1,\infty]$ and $s\in[0,1]$. A finite set $A\subset\mathbb{F}_q^d$
is called a \emph{$(u,s)$-Salem set} if
\begin{equation}\label{Salem-set-inequality}
\|\widehat{A}\|_u \ll q^{-d}|A|^{1-s},
\end{equation}
where the implied multiplicative constant is independent of our choice of $q$ and $A$.
\end{definition}

The $(u,s)$-Salem condition admits an equivalent formulation in terms of higher-order additive energies, as can be seen in \cite{CGKPTZ25}.
Recall that the additive energy of a finite set quantifies how often a given sum is represented. Specifically, the $2k$-fold additive energy of a finite set
$A\subset\mathbb{F}_q^d$ is defined to be
\[
\Lambda_{2k}(A)
:= \bigl|\{(x_1,\dots,x_{2k})\in A^{2k}:
x_1+\cdots+x_k = x_{k+1}+\cdots+x_{2k}\}\bigr|.
\]
Combining these notions led to the following equivalence, taken from \cite{CGKPTZ25}.


\begin{lemma}\label{cor:equivalence-Salem-bound}
Let $s\in[0,1]$ and $k\ge1$. The $(2k,s)$-Salem condition
\[
\|\widehat{A}\|_{2k} \ll q^{-d}|A|^{1-s}
\]
is equivalent to the additive energy bound
\[
\Lambda_{2k}(A)
\ll
\frac{|A|^{2k}}{q^d}
+
|A|^{2k(1-s)}.
\]
\end{lemma}
Since Lemma~\ref{cor:equivalence-Salem-bound} provides a direct conversion mechanism, we will work directly with additive energy estimates. To this end, we adopt the following definition.
\begin{definition}\label{(4,s)-Salem-set}
Let $s > 0$. A finite set $A \subset \mathbb{F}_q^d$ is called a \emph{$(4,s)$-Salem set} if
\[
\Lambda_4(A) \ll |A|^{4 - 4s}+\frac{|A|^4}{q^d},
\]
where the implied constant may depend on $s$ and $d$ but not on $q$ or $A$.
\end{definition}

Let $\|x\| := x_1^2 + \cdots + x_d^2$ for $x \in \mathbb{F}_q^d$.

In this paper, we are interested in an incidence structure associated with $(4, s)$-Salem sets, namely, incidences between points and spheres in $\mathbb{F}_q^d$, where the point set is a $(4, s)$-Salem set. Beyond their geometric interest, finite field incidence bounds have found numerous applications in additive combinatorics, pseudorandomness and extractor theory, restriction theory, and related topics in theoretical computer science, see \cite{bkt, Zdvir, KPV21, Lewko19, MT04, MSS1}.

For arbitrary point sets, due to degenerate configurations, incidence bounds might be weak without additional assumptions. In contrast, random point sets admit strong incidence bounds, but their lack of explicit structure limits the applicability of such results.

The $(4, s)$-Salem setting provides a natural framework for bridging this gap. The parameter $s \in [0,1]$ measures how arithmetically \textit{spread out} the set is: larger values of $s$ indicate lower additive energy and behavior closer to that of a random set. This intermediate structure allows us to obtain incidence bounds that interpolate between the extremes.

Let $P \subset \mathbb{F}_q^d$ be a set of points and let $S$ be a set of spheres. We denote by $I(P,S)$ the number of incidences:
\[
I(P,S) := \bigl|\{(x,\sigma) \in P \times S : x \text{ lies on } \sigma\}\bigr|.
\]
In this paper, a sphere centered at $a\in \mathbb{F}_q^d$ of radius $r\in \mathbb{F}_q$ is defined by 
\[\{x\in \mathbb{F}_q^d\colon ||x-a||=r\}.\]

When $P$ and $S$ are arbitrary sets, the following theorem was proved by Cilleruelo, Iosevich, Lund, Roche-Newton, and Rudnev \cite{CILRR}, and independently by Phuong, Pham, and Vinh \cite{PPV} by using an elementary counting argument and a spectral graph theory method, respectively,

\begin{theorem}\label{Thm-incidence-spheres}
    Let $P$ be a set of points in $\mathbb{F}_q^d$ and $S$ be a set of spheres in $\mathbb{F}_q^d$. Then, we have 
    \[ \left|I(P,S)-\frac{|P| \, |S|}{q} \right| \leq q^{\frac{d}{2}} |P|^{\frac{1}{2}} \, |S|^{\frac{1}{2}}.\]
\end{theorem}

In \cite{KLP22}, Koh, Lee, and Pham introduced a clever approach that connects this incidence problem with cone extension estimates. As a consequence, they obtained the following improvement. We will comment more on their framework in the last section of this paper.
\begin{theorem}[Koh-Lee-Pham, \cite{KLP22}]\label{Thm-incidence-Q-spheres}
Let $P$ be a set of points in $\mathbb{F}_q^d$ and $S$ be a set of spheres in $\mathbb{F}_q^d$. 
\begin{enumerate}
\item[(i)]
If $d \equiv 2 \pmod{4}$, $q \equiv 3 \pmod{4}$, and
$|S| \le q^{\frac{d}{2}}$, then we have
\[
\left|
I(P,S)
-
\frac{|P| \, |S|}{q}
\right|
\le
q^{\frac{d-1}{2}}
|P|^{\frac12}|S|^{\frac{1}{2}}.
\]

\item[(ii)]
If $d \equiv 0 \pmod{4}$, or $d$ is even and $q \equiv 1 \pmod{4}$, then the same conclusion holds under the condition
\[
|S| \le q^{\frac{d-2}{2}}.
\]

\item[(iii)]
If $d \ge 3$ is an odd integer, then the same conclusion holds under the condition
\[
|S| \le q^{\frac{d-1}{2}}.
\]
\end{enumerate}
\end{theorem}

This result has been extended to other ranges of $d$ and $q$ in \cite{Kohpham} and recently in \cite{KLXY25}. Theorems \ref{Thm-incidence-spheres} and \ref{Thm-incidence-Q-spheres} are generally optimal. Some sharpness examples can be found in \cite{KLP22} and \cite{KLXY25}.

Our main result in the $(4, s)$-Salem setting is the following.


\begin{theorem}\label{thm1}
Let $d \ge 2$ and let $P \subset \mathbb{F}_q^d$ be a $(4,s)$-Salem set for some $s \in \big(\frac{1}{4}, \frac{1}{2} \big]$. Let $S$ be a finite set of spheres in $\mathbb{F}_q^d$.
Assume that
\[
|P| \ll q^{\frac{d}{4s}}.
\]
Then the number of incidences between $P$ and $S$ satisfies
\[
\left| I(P,S) - \frac{|P| \,|S|}{q} \right|
\ll q^{\frac{d}{4}} |P|^{1-s} |S|^{\frac{3}{4}},
\]
which is optimal.
\end{theorem}


\begin{remark}
\label{rem:compare-KLXY}
When $P$ is a sufficiently small subset of $\mathbb F_q^d,$ (namely, $|P| \ll q^{\frac{d}{4s}}$), the incidence estimate in Theorem~\ref{thm1} improves upon the general point-sphere bound in Theorem~\ref{Thm-incidence-spheres} whenever
\[
|S|\ll q^{d}\,|P|^{4s-2},
\]
and upon the refined bounds of Koh, Lee, and Pham in Theorem~\ref{Thm-incidence-Q-spheres} whenever
\[
|S|\ll q^{d-2}\,|P|^{4s-2}.
\]
In particular, in the extremal Salem case $s=\tfrac12$, the improvement holds for all families of spheres with $|S|\ll q^{d}$ and $|S|\ll q^{d-2}$, respectively.
\end{remark}

For incidence estimates in finite settings, one often seeks conditions under which the incidence behavior matches what would be expected from a \emph{random} set with similar size parameters. Here, the expected number of incidences is $q^{-1}|P||S|,$ so the estimates above quantify how far from the expected value the number of incidences can stray.


\paragraph*{Sketch of the proof of Theorem~\ref{thm1}.}
We lift $P\subset\mathbb{F}_q^d$ to $P'=\{(x,\|x\|):x\in P\}\subset\mathbb{F}_q^{d+1}$ and observe that any additive quadruple in $P'$,
\[
(x_1,\|x_1\|)+(x_2,\|x_2\|)=(x_3,\|x_3\|)+(x_4,\|x_4\|),
\]
forces $x_1+x_2=x_3+x_4$ and $\|x_1\|+\|x_2\|=\|x_3\|+\|x_4\|$, hence comes from an additive quadruple in $P$. 
In particular, lifting cannot create new additive quadruples, so we have the trivial upper bound $\Lambda_4(P')\le\Lambda_4(P)\ll|P|^{4-4s}=|P'|^{4-4s}$, and thus $P'$ is also a $(4,s)$–Salem set in $\mathbb{F}_q^{d+1}$. 
Next we rewrite the sphere $\|x-a\|=r$ as an affine hyperplane $(-2a,1)\cdot(x,t)=r-\|a\|$ in $\mathbb{F}_q^{d+1}$, so that $x\in P$ and $\|x-a\|=r$ if and only if $(x,\|x\|)\in P'\cap H_{a,r}$, where $H_{a,r}:= \{(x,t) \in \mathbb{F}_q^{d+1} : (-2a,1) \cdot (x,t) = r - \|a\|\}$. 
This identifies $I(P,S)$ with a point–hyperplane incidence number $I(P',H)$ in $\mathbb{F}_q^{d+1}$. We then split the spheres into those with $r-\|a\|\neq 0$ and those with $r=\|a\|$. 

For the first class, we apply the point-hyperplane incidence bound for $(4,s)$–Salem sets from Section~\ref{section2} to $P'$ and the corresponding family of hyperplanes, obtaining the main term $\frac{|P| |S|}{q}$ with an error of size $q^{\frac{d}{4}}|P|^{1-s}|S|^{\frac{3}{4}}$. 

For the second class, the equation $\|x-a\|=\|a\|$ is equivalent to $-2x\cdot a+\|x\|=0$, and we encode these incidences via the symmetric incidence estimate of Section~\ref{section2} applied to $P'$ and a suitable set of normal vectors; this yields an error term of the same form. Adding the two contributions and using the size condition on $|P|$ to ensure that the error is smaller than the main term gives the desired estimate.

The following result is a direct consequence of the incidence bound. It can be seen as the unit distance problem analog of the results on distinct distances given in \cite{iosevich-Rudnev-07} and \cite{fraser2}.


\begin{corollary}[Unit distances in $(4,s)$-Salem sets]\label{cor:unit-distances}
    Let $d \ge 2$ and let $P \subset \mathbb{F}_q^d$ be a $(4,s)$-Salem set for some $s \in (\frac{1}{4}, \frac{1}{2}]$. For any non-zero distance $r \in \mathbb{F}_q^\times$, let $N_r(P)$ denote the number of pairs $(x,y) \in P \times P$ such that $\|x-y\|=r$. If $|P| \ll q^{\frac{d}{4s}}$, then
    \[
    \left| N_r(P) - \frac{|P|^2}{q} \right| \ll q^{\frac{d}{4}} |P|^{\frac{7}{4}-s}.
    \]
\end{corollary}

\begin{proof}
    Let $S$ be the set of spheres of radius $r$ centered at points in $P$, defined as $S := \{ \sigma_a : a \in P \}$, where $\sigma_a = \{x \in \mathbb{F}_q^d : \|x-a\|=r\}$. Clearly, $|S| = |P|$.
    Observe that a pair $(x,y) \in P \times P$ satisfies $\|x-y\|=r$ if and only if $x$ is incident to the sphere $\sigma_y \in S$. Thus, $N_r(P) = I(P, S)$.
    Applying Theorem \ref{thm1} with $|S|=|P|$, the error term becomes
    \[
    \ll q^{\frac{d}{4}} |P|^{1-s} |S|^{\frac{3}{4}} = q^{\frac{d}{4}} |P|^{1-s} |P|^{\frac{3}{4}} = q^{\frac{d}{4}} |P|^{\frac{7}{4}-s}.
    \]
    This concludes the proof.
\end{proof}


\begin{remark}[Comparison with Iosevich-Rudnev's result]
    It is instructive to compare Corollary \ref{cor:unit-distances} with the classical bound for general sets established by Iosevich and Rudnev in \cite{iosevich-Rudnev-07}, which states that
    \[
    \left| N_r(P) - \frac{|P|^2}{q} \right| \lesssim q^{\frac{d-1}{2}} |P|.
    \]
    Our bound for $(4,s)$-Salem sets offers an improvement over the general bound when $|P| \ll q^{\frac{d}{4s}}$ and
    \[
    q^{\frac{d}{4}} |P|^{\frac{7}{4}-s} \ll q^{\frac{d-1}{2}} |P| \iff |P|^{\frac{3}{4}-s} \ll q^{\frac{d-2}{4}}.
    \]
    In the extreme case where $s=\frac{1}{2}$, this condition simplifies to $|P| \ll q^{d-2}$. Note that there are many examples of $(4, \frac{1}{2})$-Salem sets, and we provide an explicit construction in Section \ref{sec:explicit-sidon}.
\end{remark}
\begin{remark}
    Corollary \ref{cor:unit-distances} implies that, for a $(4, s)$-Salem set $P$ in $\mathbb{F}_q^d$, if $q^{\frac{d+4}{4s+1}}\ll |P| \ll q^{\frac{d}{4s}}$, then $\Delta(P)=\mathbb{F}_q$. So, if $s=\frac{1}{2}$, then the condition $|P|\gg q^\frac{d+4}{3}$ is sufficient. 
    If we only want to cover a positive proportion of all distances, then a better exponent of $\min \big\{ \frac{d+2}{4s+1}, \frac{d+4}{8s} \big\}$ is obtained in \cite{CGKPTZ25}. When $s=\frac{1}{4}$, i.e. the set is arbitrary in $\mathbb{F}_q^d$, the distance problem has received much attention during the last two decades. 
    The optimal condition in odd dimensions is known to be $|P|\gg q^{\frac{d+1}{2}}$ in \cite{HIKR11}. 
    In two dimensions, the state-of-the-art exponents of this problem  are $\frac{5}{4}$ and $\frac{4}{3}$, which can be found in \cite{CEHIK12, MPPRS} over prime fields and arbitrary finite fields, respectively. We refer the reader to a recent paper \cite{PY25} for some surprising applications of this topic in intersection patterns and expanding phenomena.
\end{remark}

This viewpoint can also be used to prove results in additive combinatorics. We discuss two ways to relax the notion of a Sidon set, and prove some basic results about each. We also prove a sum-product application for $(u,s)$-Salem sets $A$. Specifically, in Theorem \ref{thm:sum-product-salem} below, we get estimates involving sumsets $A+A$ and the $d$-fold sumset of the set of squares of elements of $A.$

The paper is organized as follows.
In Section~\ref{section2}, we prove the main theorem, Theorem~\ref{thm1}, by applying incidence bounds.
Section~\ref{sec:sharpness} is devoted to the construction of examples demonstrating the sharpness of the exponents appearing in Theorem~\ref{thm1}.
In Section~\ref{sec:us-salem}, by employing the same Fourier moment method, we extend the point-hyperplane and point-sphere incidence bounds from
$(4,s)$-Salem sets to $(u,s)$-Salem sets for even exponents $u=2k\ge 4$.
Finally, Section~\ref{sec:s-sidon} has the applications of the main results to additive combinatorics, as well as a brief discussion on some observations related to the extension estimates for cones from Koh, Lee, and and Pham in \cite{KLP22}.

\section{Proof of Theorem \ref{thm1}}\label{section2}

We first recall the incidence bounds from \cite{CGKPTZ25} that we will use.
\begin{theorem}[{\cite[Theorem 41]{CGKPTZ25}}]\label{nonzero}
Let $d \ge 1$, and let $P \subset \mathbb{F}_q^{d+1}$ be a $(4,s)$-Salem set. Let $H$ be a finite set of hyperplanes in $\mathbb{F}_q^{d+1}$ of the form
\[
{a} \cdot {x} = b,
\]
with ${a} \in \mathbb{F}_q^{d+1} \setminus \{0\}$ and $b \in \mathbb{F}_q^\times$. Then
\[
\left| I(P,H) - \frac{|P| \, |H|}{q} \right|
\le |H|^{\frac{3}{4}} q^{\frac{d}{4}} |P|^{1-s}.
\]
\end{theorem}

\begin{remark}
If $H$ is allowed to contain hyperplanes with $b=0$, the proof in \cite{CGKPTZ25} yields a slightly weaker bound,
\[
\left| I(P,H) - \frac{|P| \, |H|}{q} \right|
\le |H|^{\frac{3}{4}} q^{\frac{d+1}{4}} |P|^{1-s},
\]
which is still sufficient for our purposes below.
\end{remark}
The next lemma is a more symmetric counting version, also following from \cite{CGKPTZ25}.


\begin{lemma}[{\cite[Lemma 37]{CGKPTZ25}}]\label{lm-counting}
Let $U \subset \mathbb{F}_q^{d+1}$ be a $(4,s)$-Salem set. Let $U'$ be a finite set of points
\[
U' \subset (\mathbb{F}_q^{d+1} \setminus \{0\}) \times \mathbb{F}_q,
\]
and define $N(U,U')$ to be the number of pairs $((x),({a},b)) \in U \times U'$ such that
\[
{a} \cdot {x} = b.
\]
Then
\[
\left| N(U,U') - \frac{|U| \, |U'|}{q} \right|
\le q^{\frac{d+1}{4}} |U|^{1-s} |U'|^{\frac{3}{4}}.
\]
\end{lemma}
We now pass from $P \subset \mathbb{F}_q^d$ to a subset of $\mathbb{F}_q^{d+1}$ by lifting along the quadratic form.
\begin{lemma}\label{lem:lift-salem}
Let $P \subset \mathbb{F}_q^d$ be a $(4,s)$-Salem set 
and assume in addition that $|P| \ll q^{\frac{d}{4s}}$.
Define
\[
P' := \{ (x,\|x\|) : x\in P \} \subset \mathbb{F}_q^{d+1}.
\]
Then $P'$ is a $(4,s)$-Salem set in $\mathbb{F}_q^{d+1}$.
\end{lemma}

\begin{proof}
By definition,
\[
\Lambda_4(P') = \bigl|\{(x_1,x_2,x_3,x_4) \in P^4 :
(x_1,\|x_1\|) + (x_2,\|x_2\|) = (x_3,\|x_3\|) + (x_4,\|x_4\|) \}\bigr|.
\]
The equality
\[
(x_1,\|x_1\|) + (x_2,\|x_2\|) = (x_3,\|x_3\|)
+ (x_4,\|x_4\|)
\]
is equivalent to the system
\[
x_1 + x_2 = x_3 + x_4, \qquad
\|x_1\| + \|x_2\| = \|x_3\| + \|x_4\|.
\]
Hence any such quadruple in $P'$ is, in particular, an additive energy quadruple of $P$. Therefore,
\[
\Lambda_4(P') \le \Lambda_4(P).
\]
Since $P$ is $(4,s)$-Salem, it follows that
\[
\Lambda_4(P') \ll |P|^{4-4s} = |P'|^{4-4s}.
\]
This shows that $P'$ is also $(4,s)$-Salem.
\end{proof}


\subsection*{Reduction to point-hyperplane incidences}
Let $S$ be a set of spheres of the form
\[
\|x - a'\| = r, \qquad a'\in \mathbb{F}_q^d,\ r \in \mathbb{F}_q.
\]
For each $(a',r)$, consider the corresponding affine equation
\[
\|x- a'\| = r.
\]
Expanding the quadratic form yields
\[
\|x\| - 2 x \cdot a' + \|a'\| = r,
\]
or equivalently
\[
-2 x\cdot a' + \|x\| = r - \|a'\|.
\]
If we write points of $\mathbb{F}_q^{d+1}$ as $(x,t)$ with $x \in \mathbb{F}_q^d$ and $t \in \mathbb{F}_q$, then the last equation becomes
\[
(-2a',1) \cdot (x,t) = r - \|a'\|,
\]
where the dot product is taken in $\mathbb{F}_q^{d+1}$. Thus each sphere $\|x-a'\|=r$ corresponds to a hyperplane
\[
H_{a',r} := \{(x,t) \in \mathbb{F}_q^{d+1} : (-2a',1) \cdot (x,t) = r - \|a'\|\}.
\]
The incidence relation
\[
x \in P,\quad \|x-a'\| = r
\]
is equivalent to
\[
(x,\|x\|) \in P',\quad (x,\|x\|) \in H_{a',r}.
\]
Hence if we set
\[
H := \{ H_{a',r} : \|x-a'\| = r \in S \},
\]
then
\[
I(P,S) = I(P',H).
\]
We now split $S$ according to whether the right-hand side $r - \|a'\|$ vanishes or not. Define
\[
S_1 := \{ \|x -a'\| = r \in S : \|a'\| - r \neq 0 \},
\qquad
S_2 := S \setminus S_1.
\]
Let $H_1$ and $H_2$ be the corresponding subsets of $H$:
\[
H_1 := \{H_{a',r} : \|x -a'\| = r \in S_1\},
\qquad
H_2 := \{H_{a',r} : \|x -a'\|= r \in S_2\}.
\]
By construction, for $H_1$ we have
\[
(-2a',1) \cdot (x,t) = b, \quad b := r - \|a'\| \in \mathbb{F}_q^\times,
\]
so all hyperplanes in $H_1$ have nonzero right-hand side. For $H_2$ we have $b = 0$.

We clearly have
\[
I(P,S) = I(P,S_1) + I(P,S_2) = I(P',H_1) + I(P',H_2).
\]


We treat the two cases separately, according to whether the associated hyperplanes have zero or nonzero constant term.

We first consider the spheres in $S_1$ which satisfy $\|a'\|-r \neq 0$.
By Lemma \ref{lem:lift-salem}, the lifted set $P'$ is a $(4,s)$-Salem subset of $\mathbb{F}_q^{d+1}$. Since every hyperplane in $H_1$ has nonzero right-hand side, we may apply Theorem \ref{nonzero} with $P$ replaced by $P'$ and $H$ replaced by $H_1$ to obtain
\[
\left| I(P',H_1) - \frac{|P'| \, |H_1|}{q} \right|
\le q^{\frac{d}{4}} |P'|^{1-s}  |H_1|^{\frac{3}{4}}.
\]
Because $|P'|=|P|$ and $|H_1| = |S_1|$, this yields
\[
\left| I(P,S_1) - \frac{|P| \, |S_1|}{q} \right|
\le q^{\frac{d}{4}} |P|^{1-s}  |S_1|^{\frac{3}{4}}.
\]



Now consider the spheres in $S_2$, which satisfy $\|a'\| = r$. For such a sphere we have
\[
\|x -a'\| = r \quad\Longleftrightarrow\quad -2 x \cdot a' + \|x\| = 0.
\]
Equivalently,
\[
(x,\|x\|) \cdot (-2a',1) = 0.
\]
We will apply Lemma \ref{lm-counting} to a suitable choice of $U$ and $U'$. Let
\[
U := P' = \{(x,\|x\|) : x \in P\} \subset \mathbb{F}_q^{d+1}.
\]
For each sphere $\|x-a'\| = r \in S_2$ and each $\lambda \in \mathbb{F}_q^\times$, consider the vector
\[
{a}_{a',\lambda} := (-2\lambda a',\lambda) \in \mathbb{F}_q^{d+1},
\]
and set
\[
U' := \{ ({a}_{a',\lambda},0) : \|x-a'\| = r \in S_2,\ \lambda \in \mathbb{F}_q^\times\}.
\]
Since $\lambda \neq 0$ and $a$ is arbitrary, each ${a}_{a',\lambda}$ is nonzero, so indeed
\[
U' \subset (\mathbb{F}_q^{d+1} \setminus \{0\}) \times \mathbb{F}_q.
\]
Moreover,
\[
|U'| = (q-1)|S_2|.
\]
An ordered pair $((x,\|x\|),({a}_{a',\lambda},0)) \in U \times U'$ satisfies
\[
{a}_{a',\lambda} \cdot (x,\|x\|) = 0
\]
if and only if
\[
(-2\lambda a',\lambda) \cdot (x,\|x\|) = 0
\quad\Longleftrightarrow\quad
\lambda\,(-2 x \cdot a' + \|x\|) = 0.
\]
Since $\lambda \neq 0$, this is equivalent to
\[
-2 x \cdot a' + \|x\| = 0,
\]
which, as observed, is the condition that $x$ lies on the sphere $\|x-a'\| = r$ with $\|a'\| = r$.

Thus every incidence $(x,\|x-a'\| = r) \in P \times S_2$ gives rise to exactly $q-1$ pairs $((x,\|x\|),({a}_{a',\lambda},0))$ with $\lambda \in \mathbb{F}_q^\times$, and conversely every such pair comes from a unique incidence. Therefore
\[
N(U,U') = (q-1)\, I(P,S_2).
\]
We now apply Lemma \ref{lm-counting} to $U$ and $U'$. We obtain
\[
\left| N(U,U') - \frac{|U| \, |U'|}{q} \right|
\le |U'|^{\frac{3}{4}} q^{\frac{d+1}{4}} |U|^{1-s}.
\]
Substituting $|U| = |P|$ and $|U'| = (q-1)|S_2|$, and using $N(U,U') = (q-1) I(P,S_2)$, we get
\[
\left| (q-1) I(P,S_2) - \frac{|P|(q-1)|S_2|}{q} \right|
\le (q-1)^{\frac{3}{4}} |S_2|^{\frac{3}{4}} q^{\frac{d+1}{4}} |P|^{1-s}.
\]
Dividing both sides by $q-1$ yields
\[
\left| I(P,S_2) - \frac{|P| \, |S_2|}{q} \right|
\le (q-1)^{-\frac{1}{4}} |S_2|^{\frac{3}{4}} q^{\frac{d+1}{4}} |P|^{1-s} \ll q^{\frac{d}{4}} |S_2|^{\frac{3}{4}} |P|^{1-s}.
\]



Combining the bounds for $S_1$ and $S_2$, we have
\[
\begin{aligned}
\left| I(P,S) - \frac{|P| \, |S|}{q} \right|
&= \left| \left(I(P,S_1) - \frac{|P| \, |S_1|}{q}\right)
    + \left(I(P,S_2) - \frac{|P| \, |S_2|}{q}\right) \right| \\
&\le \left| I(P,S_1) - \frac{|P| \, |S_1|}{q} \right|
     + \left| I(P,S_2) - \frac{|P| \, |S_2|}{q} \right| \\
&\ll q^{\frac{d}{4}} |P|^{1-s}
   \left( |S_1|^{\frac{3}{4}} + |S_2|^{\frac{3}{4}} \right).
\end{aligned}
\]
Since $|S_1|, |S_2| \le |S|$, we obtain
\[
\left| I(P,S) - \frac{|P| \, |S|}{q} \right|
\ll q^{\frac{d}{4}} |P|^{1-s} |S|^{\frac{3}{4}}.
\]
This completes the proof of Theorem \ref{thm1}.

\section{Sharpness example}\label{sec:sharpness}

We show that the exponents in Theorem~\ref{thm1} are sharp (up to constants) for $s \in \big( \frac{1}{4}, \frac{1}{2} \big]$. The example exploits a large totally isotropic subspace $W\le \mathbb{F}_q^d$ for the quadratic
form $\|x\|$. We choose a subset $P\subset W$ with controlled additive energy, guaranteeing the $(4,s)$-Salem condition, and take $S$ to be the family of zero-radius spheres centered at points of $W$. Since $\|x-a\|=0$ for all $x,a\in W$, every point is incident to every sphere, giving $I(P,S)=|P||S|$ and matching the error term.


\subsection{Sharpness examples of the incidence theorem}
Throughout this section, assume that $q$ is odd.
Recall that $\|x\|:=x_1^2+\cdots+x_d^2$ and $x\cdot y:=\sum\limits_{j=1}^d x_jy_j$.

\begin{lemma}\label{lem:exist-isotropic}
Assume that $q$ is odd and that either $(d=4k$, $k \in \mathbb{N})$ or $(d=4k+2$, $k \in \mathbb{N}, ~q \equiv 1 \pmod {4})$.
Then there exists a subspace $W\le \mathbb{F}_q^{\,d}$ of dimension $\frac d2$ such that
\[
\|w\|=0 \quad \text{for all } w\in W,
\qquad\text{and}\qquad
x\cdot y=0 \quad \text{for all } x,y\in W.
\]
In particular, $W$ is totally isotropic for the quadratic form $\|\cdot\|$.
\end{lemma}

\begin{proof}
We begin by following part of the proof of Lemma 5.1 from \cite{HIKR11}, which holds under our assumptions on $q$ and $d$.

We first consider the case $q \equiv 1 \pmod 4$. Then, there is an element of order 4 in the multiplicative group, $\sqrt{-1},$ and so there exist $\frac d2$ vectors
$v_1,\dots,v_{d/2}\in \mathbb{F}_q^{\,d}$ where $v_\ell$ has 1 for the entry with index $2\mathrm{i}-1$ and $\sqrt{-1}$ for the entry at index $2\mathrm{i}$. Therefore,
\[
\|v_\ell \|=0 \quad \text{for every } \ell,
\qquad\text{and}\qquad
v_\ell\cdot v_j=0 \quad \text{for all } \ell\neq j.
\]
Define
\[
W:=\mathrm{span}_{\mathbb{F}_q}(v_1,\dots,v_{d/2}).
\]
Since the vectors $v_1,\dots,v_{d/2}$ are linearly independent, we have $\dim W=\frac d2$.

Let $w=\sum\limits_{\ell=1}^{d/2} c_\ell v_\ell\in W$.
Using $v_\ell\cdot v_j=0$ for $i\neq j$ and $\|v_\ell\|=v_\ell\cdot v_\ell=0$, we obtain
\[
\|w\|
=
w\cdot w
=
\sum_{\ell,j} c_\ell c_j (v_\ell\cdot v_j)
=
\sum_{\ell=1}^{d/2} c_\ell^2 (v_\ell \cdot v_\ell)
=
0.
\]
Similarly, if $x=\sum\limits_\ell a_\ell v_\ell$ and $y=\sum\limits_j b_j v_j$ belong to $W$, then
\[
x\cdot y
=
\sum_{\ell,j} a_\ell b_j (v_\ell\cdot v_j)
=
0.
\]
This proves the case $q \equiv 1 \pmod 4$.

We now consider the case $d=4k$ and $q\equiv 3 \pmod 4$.
Since $q$ is odd, the equation
\[
a^2+b^2=-1
\]
admits a solution $(a,b)\in\mathbb F_q^2$.
Fix such a pair $(a,b)$.

We decompose the coordinate indices $\{1,\dots,d\}$ into $k$ disjoint blocks of size $4$.
On each block, we define two vectors supported only on the corresponding four coordinates.
More precisely, for $j=0,1,\dots,k-1$, define
\[
v_{2j+1}
=
(\underbrace{0,\dots,0}_{4j},
1,0,a,b,
\underbrace{0,\dots,0}_{d-4j-4}),
\]
\[
v_{2j+2}
=
(\underbrace{0,\dots,0}_{4j},
0,1,b,-a,
\underbrace{0,\dots,0}_{d-4j-4}).
\]
A direct computation shows that each such vector is isotropic:
\[
\|v_\ell\| = 1^2 + a^2 + b^2 = 0.
\]
Moreover, within each block we have
\[
v_{2j+1}\cdot v_{2j+2}=0,
\]
while vectors supported on different blocks are automatically orthogonal with respect to the dot product.
Hence
\[
v_\ell\cdot v_{\ell'}=0
\quad \text{ for all } 1\le \ell, \ell'\le \tfrac d2.
\]

Let
\[
W:=\mathrm{span}_{\mathbb F_q}\{v_1,\dots,v_{d/2}\}.
\]
The vectors $v_1,\dots,v_{d/2}$ are linearly independent, so
$\dim W=d/2$.
Since the dot product vanishes identically on $W$, every vector $w\in W$ satisfies
\[
\|w\|=w\cdot w=0.
\]
Therefore, $W$ is a totally isotropic subspace of dimension $\frac{d}{2}$, which completes the proof in the case $d=4k$ and $q\equiv3\pmod4$.

\end{proof}


\begin{lemma}\label{lem:salem-subset-W}
Assume that either $(d=4k$, $k \in \mathbb{N})$ or $(d=4k+2$, $k \in \mathbb{N}, ~q \equiv 1 \pmod {4})$. Let $W\le \mathbb{F}_q^{\,d}$ be a subspace as in Lemma~\ref{lem:exist-isotropic}.
Write $M:=|W|=q^{\frac d2}$.
Fix $s\in(\frac14,\frac12]$ and set
\[
N:=\left\lfloor q^{\frac{d}{8s}}\right\rfloor. 
\]
Then there exists a set $P\subset W$ with $|P|=N$ such that
\[
\Lambda_4(P)\le C\,N^{4-4s},
\]
where $C>0$ is an absolute constant.
In particular, $P$ is a $(4,s)$-Salem set in the sense of Definition~\ref{(4,s)-Salem-set}.
\end{lemma}

\begin{proof}
Let $M:=|W|$ and let $N$ be as in the statement. 
Choose $P$ uniformly at random among all $N$-element subsets of $W$.
For $x\in W$, write ${\bf 1}_P(x)$ for the indicator that $x\in P$.

For $x_1,x_2,x_3\in W$, set $x_4:=x_1+x_2-x_3\in W$. Then
\begin{equation}\label{eq:Lambda4-expand}
\Lambda_4(P)
=\sum_{x_1,x_2,x_3\in W}
{\bf 1}_P(x_1)\,{\bf 1}_P(x_2)\,{\bf 1}_P(x_3)\,{\bf 1}_P(x_4).
\end{equation}
Taking expectation and using linearity of expectation gives
\begin{equation}\label{eq:ELambda4}
\mathbb{E}\,\Lambda_4(P)
=\sum_{x_1,x_2,x_3\in W}
\mathbb{P}\bigl(\{x_1,x_2,x_3,x_4\}\subset P\bigr).
\end{equation}

Fix $(x_1,x_2,x_3)\in W^3$ and let 
\[
k:=\bigl|\{x_1,x_2,x_3,x_4\}\bigr|
\]
be the number of distinct elements among these four points. 
Since $P$ is a uniformly random $N$-subset of a set of size $M$, we have
\[
\mathbb{P}\bigl(\{x_1,x_2,x_3,x_4\}\subset P\bigr)
=\frac{(N)_k}{(M)_k}
\le \Bigl(\frac{N}{M}\Bigr)^k,
\]
where $(u)_k=u(u-1)\cdots(u-k+1)$ is the falling factorial.

Now we bound the number of triples according to whether $k=4$ or $k\le 3$.

\smallskip
\noindent
\emph{(i) The case $k=1$.}
This happens iff $x_1=x_2=x_3$, in which case also $x_4=x_1$. 
So, there are exactly $M$ such triples, and for each,
\[
\mathbb{P}(\{x_1\}\subset P)=\frac{N}{M}.
\]
Hence, the total contribution of $k=1$ triples to \eqref{eq:ELambda4} is at most $M\cdot (N/M)=N$.

\smallskip
\noindent
\emph{(ii) The case $2\le k\le 3$.}
If $k\le 3$, then at least two of $x_1,x_2,x_3$ coincide.
Indeed, if $x_4$ coincides with one of $x_1,x_2,x_3$, then one checks that this forces
$x_1=x_3$ or $x_2=x_3$ or $x_1=x_2$.
In any case, $k\le 3$ implies
\[
x_1=x_2 \quad\text{or}\quad x_1=x_3 \quad\text{or}\quad x_2=x_3.
\]
Each of these equalities describes at most $M^2$ triples, so the total number of triples with $2\le k\le 3$
is at most $3M^2$. For these triples we use the crude bound $k\ge 2$, hence
\[
\mathbb{P}(\{x_1,x_2,x_3,x_4\}\subset P)\le \Bigl(\frac{N}{M}\Bigr)^2,
\]
so their total contribution is at most $3M^2\cdot (N/M)^2=3N^2$.

\smallskip
\noindent
\emph{(iii) The case $k=4$.}
There are at most $M^3$ triples $(x_1,x_2,x_3)$ in total, and for each such triple
\[
\mathbb{P}(\{x_1,x_2,x_3,x_4\}\subset P)\le \Bigl(\frac{N}{M}\Bigr)^4.
\]
So the total contribution of the $k=4$ case is at most
\[
M^3\Bigl(\frac{N}{M}\Bigr)^4=\frac{N^4}{M}.
\]

Combining (i)–(iii) in \eqref{eq:ELambda4} gives
\[
\mathbb{E}\,\Lambda_4(P)\ \le\ \frac{N^4}{M}+3N^2+N
\ \le\ \frac{N^4}{M}+4N^2.
\]
Since $M=q^{d/2}$ and $N=q^{d/(8s)}$, we have $M=N^{4s}$, hence
\[
\frac{N^4}{M}=N^{4-4s}.
\]
Because $s\le \frac12$, we have $4-4s\ge 2$, and therefore, $N^2\le N^{4-4s}$ for $N\ge 1$.
Thus,
\[
\mathbb{E}\,\Lambda_4(P)\ \le\ 5\,N^{4-4s}.
\]
In particular, there exists at least one $N$-element subset $P\subset W$ such that
\[
\Lambda_4(P)\le 5\,N^{4-4s}.
\]
This proves the lemma (with an absolute constant $C=5$).
\end{proof}

\begin{proposition}\label{prop:sharpness-general-s}
Assume that either $(d=4k$, $k \in \mathbb{N})$ or $(d=4k+2$, $k \in \mathbb{N}, ~q \equiv 1 \pmod {4})$.
Fix $s\in(\frac14,\frac12]$ and set
\[
N:=\left\lfloor q^{\frac{d}{8s}}\right\rfloor. 
\]
Then there exist a set of points $P\subset \mathbb{F}_q^{\,d}$ and a set of spheres $S$ of the form
$\|x-a\|=0$ such that:
\begin{enumerate}
\item[(i)] $|P|=N$, $|S|=q^{\frac d2}$, and $P$ is a $(4,s)$-Salem set in the sense of
Definition~\ref{(4,s)-Salem-set};
\item[(ii)] $I(P,S)=|P|\,|S|$;
\item[(iii)] consequently,
\[
\left|I(P,S)-\frac{|P|\,|S|}{q}\right|
\sim
q^{\frac{d}{4}}\,|P|^{1-s}\,|S|^{\frac{3}{4}},
\]
so the exponents of $q$, $|P|$, and $|S|$ in the error term of Theorem~\ref{thm1}
are best possible, up to absolute constants, for this range of $s$.
\end{enumerate}
\end{proposition}

\begin{proof}
Let $W$ be as in Lemma~\ref{lem:exist-isotropic}.
By Lemma~\ref{lem:salem-subset-W}, there exists $P\subset W$ with $|P|=N$ and
\[
\Lambda_4(P)\ll N^{4-4s}.
\]
Thus, $P$ is a $(4,s)$-Salem set.

Define
\[
S:=\{\sigma_a:\ a\in W\},
\qquad
\sigma_a:=\{x\in\mathbb{F}_q^{\,d}:\ \|x-a\|=0\}.
\]
Then $|S|=|W|=q^{\frac d2}$.

Let $x\in P\subset W$ and $a\in W$.
Then $x-a\in W$, so $\|x-a\|=0$ by Lemma~\ref{lem:exist-isotropic}.
Thus, $x\in \sigma_a$ for every $x\in P$ and every $a\in W$.
Therefore,
\[
I(P,S)=|P|\,|S|.
\]

Since $|S|=q^{\frac d2}$, we have
\[
\left|I(P,S)-\frac{|P|\,|S|}{q}\right|
=
|P|\,|S|\left(1-\frac1q\right)
\sim
|P|\,|S|
=
N\,q^{\frac d2}.
\]
On the other hand,
\[
q^{\frac d4}\,|P|^{1-s}\,|S|^{\frac34}
=
q^{\frac d4}\,N^{1-s}\,\bigl(q^{\frac d2}\bigr)^{\frac34}
=
q^{\frac{5d}{8}}\,N^{1-s}.
\]
Since $N=q^{\frac{d}{8s}}$, we have $N^{s}=q^{\frac d8}$, and hence
\[
q^{\frac{5d}{8}}\,N^{1-s}
=
q^{\frac{5d}{8}}\frac{N}{N^{s}}
=
q^{\frac{5d}{8}}\frac{N}{q^{\frac d8}}
=
N\,q^{\frac d2}.
\]
Thus, the two quantities are comparable, which proves (iii).
\end{proof}


\subsection{An explicit $(4,\frac12)$-Salem set}\label{sec:explicit-sidon}
This section provides an explicit example showing that the endpoint $s=\frac12$ is nonempty. In particular, this example supports the extreme case in Corollary~\ref{cor:unit-distances}. To this end, we introduce the concept of a Sidon set, or set with minimal additive energy. A subset $E$ of an abelian group (usually real numbers, integers, or the elements of a finite field under addition) is called Sidon if the only solutions to the equation $x_1+x_2 = x_3+x_4$ for $x_j\in E$ occur when $\{x_1,x_2\} = \{x_3,x_4\}.$ Here, the construction is a Sidon set of size comparable to $q^{\frac d2}$, so it satisfies
\[
\Lambda_4(E)\ll |E|^2.
\]

Let $p$ be an odd prime, let $n,d\in\mathbb{N}$, and set $q=p^n$.
Assume that $nd$ is even and define
\[
q_0:=p^{\frac{nd}{2}}=q^{\frac d2}.
\]
Then $\mathbb{F}_{q_0}^2$ and $\mathbb{F}_q^d$ are isomorphic as $\mathbb{F}_p$-vector spaces.
Fix an $\mathbb{F}_p$-linear isomorphism
\[
\psi:\ \mathbb{F}_{q_0}^2\to \mathbb{F}_q^d.
\]
Define
\[
U
:=
\psi\bigl(\{(t,t^2):\ t\in\mathbb{F}_{q_0}\}\bigr)
\subset \mathbb{F}_q^d.
\]
The following lemma is well-known and can be proved by a direct computation.

\begin{lemma}\label{lem:sidon-graph}
The set $\{(t,t^2):\ t\in\mathbb{F}_{q_0}\}$ is Sidon in the additive group $\mathbb{F}_{q_0}^2$.
\end{lemma}
The next lemma shows that $U$ is $(4, \frac{1}{2})$-Salem set.
\begin{lemma}\label{lem:U-size-energy}
The set $U$ satisfies $|U|=q^{\frac d2}$ and
\[
\Lambda_4(U)\le 2|U|^2.
\]
In particular, $U$ is $(4, \frac{1}{2})$-Salem set.
\end{lemma}

\begin{proof}
Since $\psi$ is bijective, we have $|U|=|\mathbb{F}_{q_0}|=q_0=q^{\frac d2}$.
Lemma~\ref{lem:sidon-graph} shows that the preimage of $U$ is a Sidon set in $\mathbb{F}_{q_0}^2$.
Since $\psi$ is an additive isomorphism, $U$ is Sidon in $\mathbb{F}_q^d$.

Write
\[
r_{U-U}(t)=|\{(x,y)\in U^2:\ x-y=t\}|.
\]
Then $r_{U-U}(0)=|U|$, and $r_{U-U}(t)\le 1$ for $t\ne 0$.
Therefore,
\[
\Lambda_4(U)=\sum_{t\in\mathbb{F}_q^d} r_{U-U}(t)^2
\le
|U|^2+\sum_{t\ne 0} r_{U-U}(t)
=
|U|^2+(|U|^2-|U|)
\le
2|U|^2.
\]
This proves the lemma.
\end{proof}

\section{Extensions with $(u,s)$-Salem sets}\label{sec:us-salem}

In this section we extend the point-hyperplane and point-sphere incidence bounds from
$(4,s)$-Salem sets to $(u,s)$-Salem sets for even exponents $u=2k\ge 4$.
The argument follows the same Fourier moment method as in \cite{CGKPTZ25}, and it is an
even-moment variant of the proof used in Section~\ref{section2}.
We present the details in a form that keeps the notation consistent with the rest of the paper. Throughout this section, let $u=2k$ with $k\ge 2$.

\begin{lemma}\label{lem:incidence-us}
Let $d\ge 2$ and let $u=2k$ with $k\ge 2$.
Let $P\subset \mathbb{F}_q^{\,d}$ be a set such that
\begin{equation}\label{eq:us-salem-fourier}
\sum_{m\in\mathbb{F}_q^d\setminus\{0\}} |\widehat{P}(m)|^{u}
\ \ll_u\
q^{-(u-1)d}\,|P|^{u(1-s)},
\end{equation}
where
\[
\widehat{P}(m):=q^{-d}\sum_{x\in\mathbb{F}_q^d} P(x)\,\chi(-m\cdot x).
\]
Let $\mathcal{T}\subset (\mathbb{F}_q^d\setminus\{0\})\times \mathbb{F}_q$, and define
\[
N(P,\mathcal{T})
:=
\bigl|\{(x,(a,b))\in P\times \mathcal{T}:\ a\cdot x=b\}\bigr|.
\]
Then
\[
N(P,\mathcal{T})
\ \le\
\frac{|P|\,|\mathcal{T}|}{q}
\ +\
C_u\,|\mathcal{T}|^{1-\frac{1}{u}}\,q^{\frac{d}{u}}\,|P|^{1-s},
\]
where $C_u>0$ depends only on $u$.
\end{lemma}

\begin{proof}
For $(a,b)\in (\mathbb{F}_q^d\setminus\{0\})\times \mathbb{F}_q$, write
\[
N(a,b):=\bigl|\{x\in P:\ a\cdot x=b\}\bigr|,
\qquad
D(a,b):=N(a,b)-\frac{|P|}{q}.
\]
Then
\[
N(P,\mathcal{T})
=\frac{|P|\,|\mathcal{T}|}{q}
+
\sum_{(a,b)\in \mathcal{T}} D(a,b).
\]
By H\"older's inequality,
\begin{equation}\label{eq:holder-step-us}
\sum_{(a,b)\in \mathcal{T}} D(a,b)
\ \le\
|\mathcal{T}|^{1-\frac1u}\Bigl(\sum_{(a,b)\in \mathcal{T}} |D(a,b)|^{u}\Bigr)^{\frac{1}{u}}.
\end{equation}

We bound the $u$-th moment by enlarging the sum to all $(a,b)$ with $a\ne 0$.
By character orthogonality,
\[
N(a,b)
=
\frac{1}{q}\sum_{t\in\mathbb{F}_q}\sum_{x\in\mathbb{F}_q^d}
P(x)\,\chi\!\bigl(-t(a\cdot x-b)\bigr).
\]
Separating the term $t=0$ gives
\begin{equation}\label{eq:Dab-formula-us}
D(a,b)=q^{d-1}\sum_{t\in\mathbb{F}_q^\ast}\chi(tb)\,\widehat{P}(ta).
\end{equation}
Therefore,
\[
\sum_{(a,b)\in \mathcal{T}}|D(a,b)|^{u}
\ \le\
\sum_{a\ne 0}\sum_{b\in\mathbb{F}_q}|D(a,b)|^{u}
=
q^{u(d-1)}\sum_{a\ne 0}\sum_{b\in\mathbb{F}_q}
\Bigl|\sum_{t\ne 0}\chi(tb)\widehat{P}(ta)\Bigr|^{u}.
\]

Fix $a\ne 0$.
Expanding the $u$-th power and summing over $b\in\mathbb{F}_q$ enforces the constraint
$t_1+\cdots+t_k=t_{k+1}+\cdots+t_{2k}$.
Using the standard inequality
$c_1\cdots c_{2k}\le \frac{1}{2k}\sum\limits_{j=1}^{2k} c_j^{2k}$ for nonnegative $c_j$,
one obtains
\[
\sum_{b\in\mathbb{F}_q}\Bigl|\sum_{t\ne 0}\chi(tb)\widehat{P}(ta)\Bigr|^{u}
\ \ll_u\
q^{u-1}\sum_{t\ne 0} |\widehat{P}(ta)|^{u}.
\]
Thus,
\[
\sum_{a\ne 0}\sum_{b\in\mathbb{F}_q}|D(a,b)|^{u}
\ \ll_u\
q^{u(d-1)}\cdot q^{u-1}\sum_{a\ne 0}\sum_{t\ne 0} |\widehat{P}(ta)|^{u}.
\]
Using the change of variables $m=ta$, the map $(a,t)\mapsto m$ from $\{a\ne 0,\ t\ne 0\}$ onto
$\{m\ne 0\}$ has fiber size $q-1$. Hence
\[
\sum_{a\ne 0}\sum_{t\ne 0} |\widehat{P}(ta)|^{u}
=
(q-1)\sum_{m\ne 0}|\widehat{P}(m)|^{u}
\ll
q\sum_{m\ne 0}|\widehat{P}(m)|^{u}.
\]
Consequently,
\[
\sum_{a\ne 0}\sum_{b\in\mathbb{F}_q}|D(a,b)|^{u}
\ \ll_u\
q^{ud}\sum_{m\ne 0}|\widehat{P}(m)|^{u}.
\]
By \eqref{eq:us-salem-fourier},
\[
\sum_{a\ne 0}\sum_{b\in\mathbb{F}_q}|D(a,b)|^{u}
\ \ll_u\
q^{ud}\cdot q^{-(u-1)d}|P|^{u(1-s)}
=
q^{d}|P|^{u(1-s)}.
\]
Combining this with \eqref{eq:holder-step-us} yields
\[
\sum_{(a,b)\in \mathcal{T}} D(a,b)
\ \ll_u\
|\mathcal{T}|^{1-\frac1u}\,q^{\frac{d}{u}}\,|P|^{1-s}.
\]
This proves the lemma.
\end{proof}

\begin{remark}\label{rem:parity}
The proof of Lemma~\ref{lem:incidence-us} uses an even moment.
When $u=2k$ is even, expanding the $2k$-th power and summing over $b$ produces a single balanced linear constraint in the $t$-variables, which is the step that converts the moment bound into an estimate in terms of $\sum\limits_{m\ne 0}|\widehat{P}(m)|^{2k}$.
For odd $u$, this mechanism does not produce a balanced constraint of the same form.
Therefore, we restrict to even exponents $u=2k$.
\end{remark}

\begin{theorem}\label{thm:us-theorem41}
Let $u=2k\ge 4$ be even.
Let $P\subset\mathbb{F}_q^{\,d}$ satisfy \eqref{eq:us-salem-fourier}.
Let $H$ be a set of hyperplanes in $\mathbb{F}_q^{\,d}$ of the form
\[
a\cdot x=b,
\qquad a\ne 0,\quad b\ne 0.
\]
Then
\[
I(P,H)
\ \le\
\frac{|P|\,|H|}{q}
\ +\
C_u\,|H|^{1-\frac1u}\,q^{\frac{d-1}{u}}\,|P|^{1-s},
\]
where $C_u>0$ depends only on $u$.
\end{theorem}

\begin{proof}
For each hyperplane $a\cdot x=b$, the equations
$a\cdot x=b$ and $(\lambda a)\cdot x=\lambda b$ are equivalent for every $\lambda\in\mathbb{F}_q^\ast$.
Define
\[
\mathcal{T}
:=
\{(\lambda a,\lambda b): (a,b)\in H,\ \lambda\in\mathbb{F}_q^\ast\}
\subset (\mathbb{F}_q^d\setminus\{0\})\times\mathbb{F}_q.
\]
Then $|\mathcal{T}|=(q-1)|H|$ and
\[
I(P,H)=\frac{1}{q-1}\,N(P,\mathcal{T}).
\]
Applying Lemma~\ref{lem:incidence-us} and dividing by $q-1$, we obtain
\[
I(P,H)
\le
\frac{|P|\,|H|}{q}
+
C_u\,(q-1)^{-\frac1u}\,|H|^{1-\frac1u}\,q^{\frac{d}{u}}\,|P|^{1-s}
\le
\frac{|P|\,|H|}{q}
+
C_u\,|H|^{1-\frac1u}\,q^{\frac{d-1}{u}}\,|P|^{1-s}.
\]
This completes the proof.
\end{proof}

\begin{remark}\label{rem:us-bzero}
If $H$ contains hyperplanes with $b=0$, then the same argument yields the weaker bound
\[
I(P,H)
\ \le\
\frac{|P|\,|H|}{q}
\ +\
C_u\,|H|^{1-\frac1u}\,q^{\frac{d}{u}}\,|P|^{1-s}.
\]
\end{remark}

\begin{theorem}\label{thm:us-spheres}
Let $d\ge 2$ and let $u=2k\ge 4$ be even.
Let $P\subset \mathbb{F}_q^{\,d}$ satisfy \eqref{eq:us-salem-fourier}.
Let $S$ be a finite set of spheres in $\mathbb{F}_q^d$ of the form
\[
\|x-a\|=r,
\qquad a\in\mathbb{F}_q^d,\ r\in\mathbb{F}_q.
\]
Assume that
\[
|P|\ \ll\ q^{\frac{d}{us}}.
\]
Then
\[
\left| I(P,S)-\frac{|P|\,|S|}{q}\right|
\ \ll_u\
q^{\frac{d}{u}}\,|P|^{1-s}\,|S|^{1-\frac1u}.
\]
\end{theorem}

\begin{proof}
The proof follows the reduction in Section~\ref{section2}.
One lifts $P$ to a subset of $\mathbb{F}_q^{d+1}$ along the quadratic form and rewrites each sphere as an affine hyperplane.
The contribution of spheres with nonzero constant term is handled by Theorem~\ref{thm:us-theorem41}.
The remaining spheres are handled by the variant in Remark~\ref{rem:us-bzero}.
Combining the two contributions gives the stated bound.
\end{proof}


\section{Further applications and discussion}\label{sec:s-sidon}

We now give some brief notes on further applications of the results here.

\subsection{A relaxed notion of Sidon sets}

We begin by defining the difference representation function $r_{E-E}(t),$ which quantifies how often a particular difference occurs. Specifically, for $d\ge 1,$ $E\subset \mathbb{F}_q^{\,d}$, and $t\in \mathbb{F}_q^{\,d}$, define
\[
r_{E-E}(t):=
\bigl|\{(x,y)\in E\times E:\ x-y=t\}\bigr|.
\]
The main results of this paper are stated for $(4,s)$-Salem sets, since the Fourier formulation converts into bounds for the fourth energy $\Lambda_4(E)$. At the same time, in the proofs the Salem input enters through quantitative control of $r_{E-E}(t)$, typically via estimates of the form $\Lambda_4(E)\ll |E|^{4-4s}$ and their direct consequences. This suggests that several steps admit a formulation in purely combinatorial terms, phrased directly in terms of $r_{E-E}$.

Motivated by this observation, we introduce two Sidon-type regularity conditions parameterized by $s$.
The first one is pointwise and the second one is distributional. Proposition~\ref{prop:exist-s-sidon} provides explicit examples, so one can see that these hypotheses are nontrivial. To state these definitions, we give some notation. Given a set $E,$ denote the maximum number of times any difference is represented in the difference set by 
\[M_E:=\max_{t\in \mathbb{F}_q^{\,d}\setminus\{0\}} r_{E-E}(t),\]
and the number of $n$-rich differences by
\[R_n(E):=\{t\in \mathbb{F}_q^{\,d}\setminus\{0\}:\ r_{E-E}(t)\ge n\}.\]


\begin{definition}\label{def:s-sidon}
Let $s\in \big[\frac{1}{4},\frac12\big]$.

\smallskip
\noindent
(i) We say that $E$ is a \emph{strong $s$-Sidon set} if $|E|\le q^{\frac{d}{4s}}$, and there exists $C\ge 1$ (independent of $q$) such that
\begin{equation}\label{eq:strong-s-sidon}
M_E \le\ C\,|E|^{\,2-4s}.
\end{equation}

\smallskip
\noindent
(ii)  We say that $E$ is a \emph{weak $s$-Sidon set} if $|E|\le q^{\frac{d}{4s}}$, and there exists $C\ge 1$ (independent of $q$) such that for every real number $n\ge 1$ one has
\begin{equation}\label{eq:weak-s-sidon}
|R_n(E)|\le\ C\,\frac{|E|^{\,4-4s}}{n^{2}}.
\end{equation}
\end{definition}

To motivate the final definition, we remind the reader of Chebychev's Inequality.
\begin{proposition}\label{chebychev}
Given a finite set $A$ and a bounded function $f:A\rightarrow \mathbb R,$ with
\[\mu := |A|^{-1}\sum_{x\in A}f(x), \text{ and } \sigma := \sqrt{|A|^{-1}\sum_{x\in A} |f(x)-\mu|^2},\]
we have that
\[|\{x\in A: |f(x)-\mu|\geq n\sigma\}|\leq \frac{|A|}{n^2}.\]
\end{proposition}

To see how this relates to the definition of weak $s$-Sidon sets, notice that Chebychev's Inequality applied to $r_{E-E}(t)$, where $t$ ranges through $E-E,$ we have that for any set $E\subset \mathbb F_q^d,$ the bound $|R_n(E)|\ll |E-E|n^{-2}.$ By using the trivial bound $|E-E|\leq |E|^2,$ this gives us that
\begin{equation}\label{trivialChebychev}
    |R_n(E)|\ll |E|^2n^{-2}.
\end{equation}
So even if we do not have an exact value for $|E-E|,$ but we do know that the set $E$ is weak $s$-Sidon for some value of $s,$ then the parameter $s$ quantifies how much we gain by satisfying \eqref{eq:weak-s-sidon} over this general bound, \eqref{trivialChebychev}.



\begin{remark}\label{rem:s-sidon-range}
The effective range of the strong $s$-Sidon condition is constrained by both trivial bounds and integrality considerations.
Indeed, since $M_E\le |E|$, the strong condition
\eqref{eq:strong-s-sidon} is vacuous for $s\le \tfrac14$.
On the other hand, when $s>\tfrac14$, the same condition implies the nontrivial growth estimate
\[
|E-E|\ge c\,|E|^{4s},
\]
and consequently forces the size restriction
\[
|E|\le C\,q^{\frac{d}{4s}}.
\]
In particular, the endpoint $s=\tfrac12$ corresponds to the classical Sidon condition: if $E$ is a strong $\tfrac12$--Sidon set with constant $C=1$, then $E$ is a Sidon set in the usual sense.
For $s>\tfrac12$, the right-hand side of \eqref{eq:strong-s-sidon} tends to zero as $|E|\to\infty$, while the representation function $r_{E-E}(t)$ takes nonnegative integer values, so no nontrivial examples can exist in this regime.

By contrast, the weak condition \eqref{eq:weak-s-sidon} remains meaningful for all $s\in[0,1]$ and should be interpreted as a distributional analogue of pointwise Sidon control. 
Moreover, their relationship makes sense in that for a given parameter $s,$ we will see that a strong $s$-Sidon set is necessarily a weak $s$-Sidon set.
\end{remark}



\begin{remark}\label{rem:s-sidon-salem}
Recall that
\[
\Lambda_4(E)
=
\bigl|\{(x_1,x_2,x_3,x_4)\in E^4:\ x_1-x_2=x_3-x_4\}\bigr|
=
\sum_{t\in \mathbb{F}_q^{\,d}} r_{E-E}(t)^2.
\]
If $E$ is strong $s$-Sidon, then $\Lambda_4(E)\ll |E|^{4-4s}$.
Indeed, recalling that $M_E=\max\limits_{t\ne 0} r_{E-E}(t)$.
Since $\sum\limits_{t\ne 0} r_{E-E}(t)=|E|^2-|E|$, we have
\[
\Lambda_4(E)
=
|E|^2+\sum_{t\ne 0} r_{E-E}(t)^2
\le
|E|^2+M_E\sum_{t\ne 0} r_{E-E}(t)
\ll
|E|^2+|E|^{2-4s}(|E|^2-|E|)
\ll
|E|^{4-4s}.
\]
In the regime $|E|\ll q^{\frac{d}{4s}}$, the term $ \frac{|E|^4}{q^d}$ is dominated by $|E|^{4-4s}$.
Then Lemma~\ref{cor:equivalence-Salem-bound} and \cite[Remark~16]{CGKPTZ25} show that this implies that $E$ is a $(4,s)$-Salem set.

Conversely, if $E$ is a $(4,s)$-Salem set with $|E|\ll q^{\frac{d}{4s}}$, then $\Lambda_4(E)\ll |E|^{4-4s}$.
Therefore, by Chebyshev's inequality applied to $r_{E-E}$, we have that $E$ satisfies \eqref{eq:weak-s-sidon}, so 
$E$ is weak $s$-Sidon.
\end{remark}

The fact that these generalized notions of Sidon sets and Salem sets interact is consistent with the literature. In \cite{CS25}, it was shown that graphs of certain highly nonlinear functions called bent functions are Salem sets. The quadratic functions studied here fit into this general picture. Moreover, in certain settings, bent functions are not available, and a class of functions called almost perfect nonlinear functions (APN functions) is often employed. There are a number of relationships between APN functions and Sidon sets. See \cite{MT25} and the references contained therein.

We now give an example to demonstrate that weak $s$-Sidon does not imply strong $s$-Sidon, so we can see that these two definitions are not redundant. To see an example of where this occurs, we construct a set that is weakly $\frac{1}{2}$-Sidon but not strongly $\frac{1}{2}$-Sidon. Fix $q$ to be a large prime, so we can treat the elements like integers modulo $q$, set $m=\lfloor q/4 \rfloor,$ and consider the set
\[F:=\{x^2:0\leq x< \sqrt q\}\cup \{x^2+m:0\leq x< \sqrt q\}.\]
Clearly $F$ is not a strong $\frac{1}{2}$-Sidon set, as the difference $m$ appears too often, but it is the only difference that occurs more than constantly many times (in fact, $m$ is a difference with unusually high multiplicity), so one can check that $F$ is a weak $\frac{1}{2}$-Sidon set, and that $\Lambda_4(F)\ll |F|^2$.

Next, we show the notion of a weak $s$-Sidon set can give nontrivial energy bounds by modifying the calculation from Remark \ref{rem:s-sidon-salem}. Specifically, if we have a set that is not strong $s$-Sidon for a given value $s$, but is weak $s$-Sidon, we may still have some nontrivial energy bounds. We state this as a flexible technical lemma, then provide a more easily stated corollary.


\begin{lemma}\label{lem:weak-tech}
    If $E$ is weak $s$-Sidon, then
    \[\Lambda_4(E)\ll M_E \, |E|^{2-2s} \, |E-E|^\frac{1}{2}.\]
\end{lemma}

Since any $t\in E-E$ can have at most one representation $x-y$ for each $x\in E,$ we get $M_E\leq |E|.$ Therefore we obtain the following corollary.

\begin{corollary}\label{cor:weak-tech}
    If $E$ is weak $s$-Sidon, then
    \[\Lambda_4(E)\ll |E|^{3-2s}|E-E|^\frac{1}{2}.\]
\end{corollary}
 We now prove the lemma.

\begin{proof}
    For any $n\geq 1$, we split the energy into three terms: the zero term, the terms with bounded representation, and the $n$-rich terms.
    \begin{align*}
        \Lambda_4(E) &= |E|^2 + \sum_{t\in (E-E)\setminus R_n(E)}r_{E-E}(t)^2 + \sum_{t\in R_n(E)}r_{E-E}(t)^2\\
        &\ll |E|^2 + n^2 |(E-E)\setminus R_n(E)| + \frac{M_E^2|E|^{4-4s}}{n^2},\\
        &\leq |E|^2 + n^2 |E-E| + \frac{M_E^2|E|^{4-4s}}{n^2},\\
    \end{align*}
    where we used the bound on the number of representations and the fact that $E$ is a weak $s$-Sidon set in the second line. Next, we balance the contributions to the energy estimate by choosing a value
    \[n = \frac{M_E^\frac{1}{2}|E|^{1-s}}{|E-E|^\frac{1}{4}}.\]
    Plugging this in yields the claimed bound.
\end{proof}

We pause to note that the choice of $n$ in the argument above could be improved to lead toward tighter bounds in the case that one knows more about the set under consideration. The calculation above is intended to be a proof of concept. Finally, we show that examples of these sets exist for a range of parameters.

\begin{proposition}\label{prop:exist-s-sidon}
Let $d\ge 1$ and let $s\in \big[\frac{1}{4}, \frac{1}{2}\big]$.
For all sufficiently large $q$ (depending on $d$ and $s$), there exists a set
$E\subset \mathbb{F}_q^{\,d}$ which is strong $s$-Sidon, and hence also weak $s$-Sidon.

More precisely:

\begin{enumerate}
    \item[(i)] If $s\in \big[\frac{1}{4}, \frac{1}{2} \big)$, then one can choose $E$ with
        \[
        |E|\sim q^{\frac{d}{4s}}
        \qquad\text{and}\qquad
        \max_{t\in \mathbb{F}_q^{\,d}\setminus\{0\}} r_{E-E}(t)\ \ll_{d,s}\ |E|^{2-4s}.
        \]
    \item[(ii)]  If $s=\frac{1}{2}$, then there exists a Sidon set
$E\subset \mathbb{F}_q^{\,d}$ of size $|E|=q^{\lfloor \frac{d}{2} \rfloor}$.
In particular, if $d$ is even, one may take $|E|=q^{\frac{d}{2}}$.
\end{enumerate}
\end{proposition}

\begin{proof}
Write $G:=\mathbb{F}_q^{\,d}$ and $Q:=|G|=q^d$.

\smallskip
\noindent
\emph{Part (i).}
Fix an integer $N$ with $N\sim Q^{\frac{1}{4s}}=q^{\frac{d}{4s}}$, and choose $E$ uniformly at random among all $N$-element subsets of $G$.
For each fixed $t\in G\setminus\{0\}$, we have $r_{E-E}(t)=|E\cap(E+t)|$.
This random variable has hypergeometric distribution with mean
\[
\mu:=\mathbb{E}\,r_{E-E}(t)=\frac{N^2}{Q}.
\]
Since $N^{4s}\sim Q$, we have
\[
\mu\sim N^{2-4s}.
\]
Because $s<\frac{1}{2}$, the quantity $\mu$ tends to infinity polynomially in $Q$.

A standard tail bound for hypergeometric random variables yields an absolute constant $c>0$ such that
\[
\mathbb{P}\bigl(r_{E-E}(t)\ge 2\mu\bigr)\le \exp(-c\mu)
\qquad\text{for every }t\ne 0.
\]
Applying the union bound over all $t\in G\setminus\{0\}$ gives
\[
\mathbb{P}\Bigl(\exists\,t\ne 0:\ r_{E-E}(t)\ge 2\mu\Bigr)
\le
(Q-1)\exp(-c\mu).
\]
For $q$ sufficiently large, the right-hand side is strictly smaller than $1$.
Therefore, there exists a set $E$ such that $r_{E-E}(t)<2\mu$ for every $t\ne 0$.
Thus,
\[
\max_{t\ne 0} r_{E-E}(t)\ \ll_{d,s}\ N^{2-4s}.
\]
Since $|E|=N$, this proves (i).

\smallskip
\noindent
\emph{Part (ii).}
Set $m:=\lfloor \frac{d}{2}\rfloor$.
Identify $\mathbb{F}_{q^m}$ with $\mathbb{F}_q^{\,m}$ as an $\mathbb{F}_q$-vector space.
Consider
\[
S:=\{(x,x^2):\ x\in \mathbb{F}_{q^m}\}\subset \mathbb{F}_{q^m}\times \mathbb{F}_{q^m}\cong \mathbb{F}_q^{\,2m}.
\]
The same calculation as in Lemma~\ref{lem:sidon-graph} shows that $S$ is Sidon.
Embedding $\mathbb{F}_q^{\,2m}$ into $\mathbb{F}_q^{\,d}$ by appending $d-2m$ zeros preserves the Sidon property.
Therefore, we obtain a Sidon set in $\mathbb{F}_q^{\,d}$ of size $q^m=q^{\lfloor \frac{d}{2}\rfloor}$.
This proves (ii).

\smallskip
Strong $s$-Sidon implies weak $s$-Sidon by Definition~\ref{def:s-sidon}.
\end{proof}

\subsection{A sum-product application via point-sphere incidences for $(4,s)$-Salem sets}
\label{sec:sum-product-salem}

In this section, we derive a sum-product type consequence of our point-sphere incidence bound.
This follows the incidence philosophy used by Koh and Pham in \cite{Kohpham}, adapted to the
$(4,s)$-Salem setting.


\begin{theorem}
\label{thm:sum-product-salem}
Let $d\ge 2$ and let $s\in\big(\frac{1}{4}, \frac{1}{2} \big]$.
Let $A\subset \mathbb F_q$ be a finite set such that
\begin{equation}\label{eq:SP-energy-A}
\Lambda_4(A)\ \le\ C_0\,|A|^{4-4s}
\qquad\text{and}\qquad
|A|\ \le\ q^{\frac{1}{4s}},
\end{equation}
where $C_0\ge 1$ is independent of $q$.
Define
\[
P:=A^d\subset \mathbb F_q^d,
\qquad
dA^2:=\{a_1^2+a_2^2+\dots+a_d^2:a_j\in A\},
\]
and let $S$ be the family of spheres
\[
S:=\{\sigma_{c,r}:\ c\in (A+A)^d,\ r\in dA^2\},
\qquad
\sigma_{c,r}:=\{x\in\mathbb F_q^d:\ \|x-c\|=r\},
\quad
\|x\|:=x_1^2+\cdots+x_d^2.
\]
Then
\begin{equation}\label{eq:SP-master}
|A|^{2d}
\ \le\
\frac{|A|^d\,|A+A|^d\,|dA^2|}{q}
\ +\
C_1\,q^{\frac{d}{4}}\,|A|^{d(1-s)}\bigl(|A+A|^d\,|dA^2|\bigr)^{\frac{3}{4}},
\end{equation}
where $C_1>0$ depends only on $d$, $s$, and the implied constant in the incidence bound of
Theorem~\ref{thm1}.
In particular, at least one of the following holds:
\begin{align}
|A+A|^d\,|dA^2| &\ \ge\ c\,q\,|A|^d, \label{eq:SP-case1}\\
|A+A|^d\,|dA^2| &\ \ge\ c\,q^{-\frac{d}{3}}\,|A|^{\frac{4d}{3}(1+s)}, \label{eq:SP-case2}
\end{align}
where $c>0$ depends only on $d$, $s$, and $C_0$.
\end{theorem}

\begin{proof}
Recall that $I(P,S)$ denotes the number of incidences between $P$ and $S$.

\smallskip
\noindent

For each ordered pair $(x,u)\in A^d\times A^d$, set
\[
c:=x+u\in (A+A)^d,
\qquad
r:=\|u\|=\sum_{j=1}^d u_j^2\in dA^2.
\]
Then $x-(x+u)=-u$, and hence $\|x-c\|=\|u\|=r$, so $x\in\sigma_{c,r}$.
Thus every pair $(x,u)$ contributes an incidence, and therefore
\begin{equation}\label{eq:SP-lower}
I(P,S)\ \ge\ |A^d|\,|A^d|\ =\ |A|^{2d}.
\end{equation}

\smallskip
\noindent

Let $P:=A^d$.
A quadruple $(x_1,x_2,x_3,x_4)\in P^4$ satisfies $x_1+x_2=x_3+x_4$ if and only if this holds in each coordinate.
Therefore
\[
\Lambda_4(P)=\Lambda_4(A)^d\le C_0^d\,|A|^{d(4-4s)}=C_0^d\,|P|^{4-4s}.
\]
Moreover, the size condition $|A|\le q^{\frac{1}{4s}}$ implies $|P|=|A|^d\le q^{ \frac{d}{4s}}$, hence
$|P|^{4s}\le q^d$ and so
\[
\frac{|P|^4}{q^d}\ \le\ |P|^{4-4s}.
\]
Consequently,
\[
\Lambda_4(P)\ \ll\ |P|^{4-4s}+\frac{|P|^4}{q^d},
\]
so $P$ is a $(4,s)$-Salem set in $\mathbb F_q^d$ in the sense of Definition~\ref{(4,s)-Salem-set}.

\smallskip
\noindent

Since $|P|=|A|^d\le q^{\frac{d}{4s}}$, we may apply Theorem~\ref{thm1} to $P$ and $S$ and obtain
\[
\left|I(P,S)-\frac{|P|\,|S|}{q}\right|
\ \ll\
q^{\frac{d}{4}}\,|P|^{1-s}\,|S|^{\frac{3}{4}}.
\]
Thus
\begin{equation}\label{eq:SP-upper}
I(P,S)
\ \le\
\frac{|P|\,|S|}{q}
\ +\
C_1\,q^{\frac{d}{4}}\,|P|^{1-s}\,|S|^{\frac{3}{4}}.
\end{equation}
Since $|P|=|A|^d$ and $|S|=|A+A|^d\,|dA^2|$, inequality \eqref{eq:SP-upper} becomes
\[
I(P,S)
\ \le\
\frac{|A|^d\,|A+A|^d\,|dA^2|}{q}
\ +\
C_1\,q^{\frac{d}{4}}\,|A|^{d(1-s)}\bigl(|A+A|^d\,|dA^2|\bigr)^{\frac{3}{4}}.
\]
Combining this with the lower bound \eqref{eq:SP-lower} gives \eqref{eq:SP-master}.

\smallskip
Finally, \eqref{eq:SP-case1} and \eqref{eq:SP-case2} follow from \eqref{eq:SP-master} by comparing the two terms
on the right-hand side with the left-hand side.
\end{proof}

A direct computation shows that, under the hypotheses $\Lambda_4(A)\ll |A|^{4-4s}$ and $|A|\le q^{\frac{1}{4s}}$, our incidence argument yields
\[
M:=\max\{|A+A|,\ |dA^2|\}
\ \gg\
\min\Bigl\{q^{\frac1{d+1}}|A|^{\frac{d}{d+1}},\ q^{-\frac{d}{3(d+1)}}|A|^{\frac{4(1+s)d}{3(d+1)}}\Bigr\}.
\]
The two terms in the minimum are equal exactly when $|A|=q^{\alpha_0}$, where
\[
\alpha_0=\frac{d+3}{d(1+4s)}.
\]
Hence, for $|A|\ge q^{\alpha_0}$ we have the lower bound
\[
M\ \gg\ q^{\frac1{d+1}}|A|^{\frac{d}{d+1}}.
\]
Moreover, one checks that
\[
q^{\frac1{d+1}}|A|^{\frac{d}{d+1}}
\ \gg\
\frac{|A|^d}{q^{\frac{d-1}{2}}}
\qquad\Longleftrightarrow\qquad
|A|\ \ll\ q^{\frac{d^2+1}{2d^2}},
\]
so in the range $q^{\max\{\frac12,\alpha_0\}}\ll |A|\ll q^{\frac{d^2+1}{2d^2}}$ our bound improves the Koh-Pham estimate
$\max\{|A+A|,|dA^2|\}\gg |A|^d/q^{ \frac{d-1}{2}}$ from \cite[Corollary~9]{Kohpham}.
Since $q^{\frac1{d+1}}|A|^{\frac{d}{d+1}}
\gg
|A|^{\frac{3d-5}{d-1}}q^{\frac{2-d}{d-1}}$ throughout this same range, it also improves the bound
$|A|^{\frac{3d-5}{d-1}}q^{\frac{2-d}{d-1}}$ in \cite{Hpham}.
Finally,
\[
q^{\frac1{d+1}}|A|^{\frac{d}{d+1}}
\ \gg\
|A|^{4s}
\qquad\Longleftrightarrow\qquad
|A|\ \ll\ q^{\frac{1}{4s(d+1)-d}},
\]
so, whenever $s>\frac{d}{4(d+1)}$ (in particular for all $s\in(\frac14,\frac12]$ and $d\ge 3$),
our bound also improves the trivial Salem consequence $|A+A|\gg |A|^{4s}$ in the range
$q^{\max \big\{\frac12,\alpha_0 \big\}}\ll |A|\ll q^{\min \big\{\frac{d^2+1}{2d^2},\,\frac{1}{4s(d+1)-d} \big\}}$.

When $\mathbb{F}_q$ is a prime field, stronger bounds are available. A key reference in this direction is the work of Pham, Vinh, and de Zeeuw in \cite{PVD}.


\subsection{On the Koh-Lee-Pham framework via cone extension estimates}
\label{subsec:KLP-framework}

In this subsection we record several observations related to the approach of Koh, Lee, and Pham in \cite{KLP22}. It seems plausible that, in the Salem setting, extension-theoretic ideas could lead to new incidence bounds. We hope to return to this question in future work.

\subsubsection*{Weighted incidences and normalization}

Let $P\subset \mathbb{F}_q^d$ be a set of points and let $S$ be a finite family of spheres in $\mathbb{F}_q^d$.
Given a weight function $w:S\to \mathbb{C}$, define the weighted incidence number
\[
I_w(P,S)
:=
\sum_{p\in P}\sum_{\sigma\in S} w(\sigma)\,{\bf 1}_{\sigma}(p).
\]
For $1\le p\le \infty$, we write
\[
\|w\|_{\ell^p(S)}
:=
\Bigl(\sum_{\sigma\in S}|w(\sigma)|^p\Bigr)^{1/p},
\qquad
\|w\|_{\ell^\infty(S)}:=\max_{\sigma\in S}|w(\sigma)|.
\]

\subsubsection*{The restriction-driven $\ell^p$ bound}

The Koh-Lee-Pham framework starts from their delicate resolution of the endpoint
extension estimate for the cone in dimension $d+2$, where $d\equiv 2\mod 4$ and $q\equiv 3\mod 4$.

Assume that an extension estimate of the form
\begin{equation}\label{eq:cone-endpoint}
R^*_{C_{d+2}}(2\to r)\ \ll\ 1,
\qquad
r=\frac{2(d+2)+4}{d+2}=\frac{2d+8}{d+2},
\end{equation}
is available.
Let $r'$ be the conjugate exponent of $r$, so that
\[
r'=\frac{r}{r-1}=\frac{2d+8}{d+6}.
\]
Then the lifting argument from spheres in $\mathbb{F}_q^d$ to the cone in $\mathbb{F}_q^{d+2}$, together with duality, yields an incidence bound of the following form:
\begin{equation}\label{eq:KLP-blackbox}
\Bigl|I_w(P,S)-q^{-1}|P|\sum_{\sigma\in S} w(\sigma)\Bigr|
\ \ll\
q^{\frac{d^2+3d-2}{2d+8}}\,
|P|^{\frac{1}{2}}\,
\|w\|_{\ell^{r'}(S)}.
\end{equation}
Equivalently,
\[
\|w\|_{\ell^{r'}(S)}
=
\Bigl(\sum_{\sigma\in S}|w(\sigma)|^{\frac{2d+8}{d+6}}\Bigr)^{\frac{d+6}{2d+8}}.
\]
The exponent
\[
\alpha_d:=\frac{d+6}{2d+8}=\frac{1}{r'}
\]
is therefore dictated by the endpoint extension exponent $r$ through conjugacy.
In particular, \eqref{eq:KLP-blackbox} is the canonical incidence estimate that one obtains from \eqref{eq:cone-endpoint} without using any additional structure of the family $S$. Thus, restriction theory provides a natural template for converting extension estimates into
incidence bounds.

\subsubsection*{A sharper $\ell^2$ bound in a sparse regime}
In several arithmetic regimes, Koh, Lee, and Pham showed that, for sufficiently small families of spheres,
the $\ell^{r'}$-norm in \eqref{eq:KLP-blackbox} can be replaced by an $\ell^2$-norm, at the cost of imposing a size restriction on $S$.

More precisely, under the corresponding hypotheses on $(d,q)$ and assuming that $|S|$ is below a dimension-dependent threshold, one has
\begin{equation}\label{eq:KLP-L2}
\Bigl|I_w(P,S)-q^{-1}|P|\sum_{\sigma\in S} w(\sigma)\Bigr|
\ \ll \
q^{\frac{d-1}{2}}\,
|P|^{\frac{1}{2}}\,
\|w\|_{\ell^2(S)},
\end{equation}
which is Theorem \ref{Thm-incidence-Q-spheres} stated in the introduction. 

Estimate \eqref{eq:KLP-L2} is stronger than \eqref{eq:KLP-blackbox} whenever it applies, since $\|w\|_{\ell^2}$
is smaller than $\|w\|_{\ell^{r'}}$ for $r'<2$.
The point is that \eqref{eq:KLP-L2} uses additional input beyond \eqref{eq:cone-endpoint}, namely a sparsity
assumption on the support of $w$ and the specific structure coming from families of spheres.

Consider the unweighted case $w\equiv 1$, then \eqref{eq:KLP-blackbox} gives
\[
\text{error}\ \ll\ q^{\frac{d^2+3d-2}{2d+8}}\ |P|^{\frac{1}{2}}\ |S|^{\alpha_d},
\qquad
\alpha_d=\frac{d+6}{2d+8}.
\]
On the other hand, the well-known point-sphere incidence bound for arbitrary sets (Theorem \ref{Thm-incidence-spheres}) gives
\[
\text{error}\ \ll\ q^{\frac{d}{2}}\ |P|^{\frac{1}{2}}\ |S|^{\frac{1}{2}}.
\]
A direct comparison shows that the restriction bound is strictly stronger precisely up to the threshold
\begin{equation*}\label{eq:KLP-threshold}
|S|\ \ll\ q^{\frac{d+2}{2}}.
\end{equation*}
Hence, \eqref{eq:KLP-blackbox} yields a genuine gain over the bounds of both Theorems \ref{Thm-incidence-spheres} and \ref{Thm-incidence-Q-spheres} in the intermediate range
\[
q^{\frac{d}{2}}\ \ll\ |S|\ \ll\ q^{\frac{d+2}{2}}.
\]

{\bf Acknowledgements} 
The authors would like to thank Thang Pham for his valuable comments, which have improved the quality of this paper. We also would like to thank Kevin O'Bryant for helpful discussions about Sidon sets. Moreover, we are grateful for the anonymous referees, who read this carefully, and provided thoughtful comments for us, which also improved the accuracy and clarity of this paper. Dung The Tran was supported by the research project QG.25.02 of Vietnam National University, Hanoi. He also would like to thank Vietnam Institute for Advanced Study in Mathematics (VIASM) for the hospitality and for the excellent working conditions.

\end{document}